\documentclass[12pt]{amsart}
\usepackage{amsmath,amssymb,amsbsy,amsfonts,amsthm,latexsym,
                        amsopn,amstext,amsxtra,euscript,amscd,color}
\usepackage[english]{babel}
\usepackage{url}
\usepackage[colorlinks,linkcolor=blue,anchorcolor=blue,citecolor=blue]{hyperref}

\begin{document}

\newtheorem{theorem}{Theorem}
\newtheorem{lemma}[theorem]{Lemma}
\newtheorem{algol}{Algorithm}
\newtheorem{cor}[theorem]{Corollary}
\newtheorem{prop}[theorem]{Proposition}

\newtheorem{proposition}[theorem]{Proposition}
\newtheorem{corollary}[theorem]{Corollary}
\newtheorem{conjecture}[theorem]{Conjecture}
\newtheorem{definition}[theorem]{Definition}
\newtheorem{remark}[theorem]{Remark}

\newcommand{\comm}[1]{\marginpar{%
\vskip-\baselineskip 
\raggedright\footnotesize
\itshape\hrule\smallskip#1\par\smallskip\hrule}}

\def\cA{{\mathcal A}}
\def\cB{{\mathcal B}}
\def\cC{{\mathcal C}}
\def\cD{{\mathcal D}}
\def\cE{{\mathcal E}}
\def\cF{{\mathcal F}}
\def\cG{{\mathcal G}}
\def\cH{{\mathcal H}}
\def\cI{{\mathcal I}}
\def\cJ{{\mathcal J}}
\def\cK{{\mathcal K}}
\def\cL{{\mathcal L}}
\def\cM{{\mathcal M}}
\def\cN{{\mathcal N}}
\def\cO{{\mathcal O}}
\def\cP{{\mathcal P}}
\def\cQ{{\mathcal Q}}
\def\cR{{\mathcal R}}
\def\cS{{\mathcal S}}
\def\cT{{\mathcal T}}
\def\cU{{\mathcal U}}
\def\cV{{\mathcal V}}
\def\cW{{\mathcal W}}
\def\cX{{\mathcal X}}
\def\cY{{\mathcal Y}}
\def\cZ{{\mathcal Z}}

\def\C{\mathbb{C}}
\def\F{\mathbb{F}}
\def\K{\mathbb{K}}
\def\Z{\mathbb{Z}}
\def\R{\mathbb{R}}
\def\Q{\mathbb{Q}}
\def\N{\mathbb{N}}
\def\M{\textsf{M}}

\def\({\left(}
\def\){\right)}
\def\[{\left[}
\def\]{\right]}
\def\<{\langle}
\def\>{\rangle}

\def\e{e}

\def\eq{\e_q}
\def\fS{{\mathfrak S}}

\def\lcm{{\mathrm{lcm}}\,}

\def\fl#1{\left\lfloor#1\right\rfloor}
\def\rf#1{\left\lceil#1\right\rceil}
\def\mand{\qquad\mbox{and}\qquad}

\def\jt{\tilde\jmath}
\def\ellmax{\ell_{\rm max}}
\def\llog{\log\log}

\def\Qbar{\overline{\Q}}
\def\GL{{\rm GL}}
\def\Aut{{\rm Aut}}
\def\End{{\rm End}}
\def\Gal{{\rm Gal}}

\title
[Congruences with intervals and subgroups]
{\bf  Congruences with intervals and subgroups modulo a prime}
\author{Marc Munsch}
\address{CRM, Universit\'e de Montr\'eal, 5357 Montr\'eal, Qu\'ebec }
\email{munsch@dms.umontreal.ca}
\author{Igor E.~Shparlinski} 
\address{Department of Pure Mathematics, University of 
New South Wales, Sydney, NSW 2052, Australia}
\email{igor.shparlinski@unsw.edu.au}

\date{\today}

\subjclass{11G07, 11L40, 11Y16}
\keywords{Character sum, large sieve} 

\begin{abstract} We obtain new results about the representation 
of almost all residues modulo a prime $p$  by a product of a small integer 
and also an element of small multiplicative subgroup 
of $(\Z/p\Z)^*$. These results are based on some ideas, and their 
modifications, of a recent work of J.~Cilleruelo and M.~Z.~Garaev (2014).
\end{abstract}

\bibliographystyle{alpha}
\maketitle

\section{Introduction}

It is well known that the progress on many classical and 
modern number theoretic questions 
depends on the existence asymptotic formulas and good upper and lower bounds
on the number of solutions to the congruences of the form
\begin{equation}
\label{eq:cong m}
au \equiv x \pmod m
\end{equation}
where $u$ runs through a multiplicative subgroup $\cG$ of
the group of units $\Z_m^*$ of the residue ring $\Z_m$ modulo an 
integers $m\ge 2$ and $x$ runs through a set $\{A+1, \ldots, A+H\}$ 
of $H$ consecutive integers, see~\cite{KoSh1} for an outline 
of such questions. In the special case when $m=p$ is a prime number and 
$\cG$ is a group of squares, this is a celebrated question 
about the distribution of quadratic residues.

Recently, various modifications of the congruence~\eqref{eq:cong m}
have been studied, such as congruences with 
elements from more general sets than subgroups on the left hand side
and also with products and ratios of 
variables from short intervals on the right hand side, 
see~\cite{BGKS1,BGKS2,
BKS1,BKS2,CillGar1,CillGar2,Gar1,Gar2,GarKar1,HarmShp,KoSh2}
and references therein. New applications of such congruences have also 
been found as well  and include questions about 
\begin{itemize}
\item   nonvanishing of 
Fermat quotients~\cite{BFKS};
\item   estimating fixed points
of the discrete logarithm~\cite{BGKS1,BGKS2};
\item distribution of pseudopowers~\cite{BKPS};
\item distribution of digits in reciprocals 
of primes~\cite{ShpSte}. 
\end{itemize}

Here we consider the congruence~\eqref{eq:cong m} in the special case 
when $m=p$ is prime. Furthermore, we are mostly interesting in the 
solvability of~\eqref{eq:cong m} for rather small intervals and 
subgroups. 

Since we consider congruences modulo primes, 
it is convenient to use the language of finite fields.

For a prime $p$ use $\F_p$ to denote the finite 
field of $p$ elements, which we assume to represented 
by the set $\{0, 1, \ldots, p-1\}$. We say that a set $\cI \subseteq \F_p$ is 
an interval of length $H$ if it contains $H$ consecutive 
elements of $\F_p$, assuming that $p-1$ is followed by $0$. 
Furthermore we say that $\cI$ is an initial interval if 
$\cI = \{1, \ldots, H\}$ (we note that it is convenient to exclude 
$0$ from initial intervals).  

Furthermore, instead of subgroups we consider a more general class sets, 
which also contain sets of $N$ consecutive powers $\{g, \ldots, g^{N}\}$ 
of a fixed element $g \in \F_p^*$. 

%

Namely, as usual for a set  $\cU \subseteq \F_p$ we use $\cU^{(m)} $ to denote its 
$m$-fold product set 
$$
\cU^{(m)} =  \{u_1\ldots u_m~:~ u_1,\ldots, u_m \in \cU\}.
$$

We say that $\cU \subseteq \F_p^*$ is an {\it approximate
subgroup\/} of $ \F_p^*$  if 
$$
\# \cU^{(2)} \le (\# \cU)^{1+o(1)},
$$
as $\# \cU \to \infty$. 

Consequently, here we study the solvability of 
equations over $\F_p$ of the type 
\begin{equation}
\label{eq:eq p}
au = x, \qquad u \in \cU, \ x \in \cI, 
\end{equation}
where $\cU \subseteq \F_p^*$  is an approximate
subgroup of $ \F_p$ and $\cI \subseteq \F_p^*$ is 
an interval. 

It has been shown by Cilleruelo and Garaev~\cite{CillGar2}
that for any $\varepsilon > 0$ there is $\delta > 0$ such that 
if $\cU = \cG$ is a subgroup of 
order $\# \cU \ge p^{3/8}$ and $\cI$ is an initial 
interval of length $\# \cI \ge p^{5/8 + \varepsilon}$ that~\eqref{eq:eq p}
has a solution for all but at most $O(p^{1-\delta})$ 
values of $a \in \F_p$. 

Here we show that the ideas of Cilleruelo and Garaev~\cite{CillGar2}
combined with the approach of Garaev~\cite{Gar2} to estimating
character sums for almost all primes, allows us to obtain similar results for a
wider range of sizes $\#\cU$ and $\# \cI$ (and also for approximate
subgroups $\cU$).  Furthermore, we use some tools from additive combinatorics
to establish a certain new  result about 
subsets of approximate subgroups, which maybe of independent 
interest. 

Throughout the paper, the implied constants in the  
symbols $O$ and $\ll$ are absolute. We recall that the 
assertions $U=O(V)$ and $U\ll V$ are both equivalent to the 
inequality $|U|\le cV$ with some constant $c$.

\section{Background on exponential and character sums}

Let $\cX_q$ denote the set of all $\varphi(q)$  multiplicative characters modulo an integer  $q\ge 2$
and let $\cX_q^*$ be the set of primitive characters $\chi\in \cX_q$,
where $\varphi(q)$ to denotes the Euler function of $q$, 
we refer to~\cite{IwKow} for a background on characters.

Let $\cA=(a_n)_{n\in\mathbb{N}}$ be an arbitrary sequence of complex numbers. 
For an integer $h$ and a character $\chi \in \cX_q$ we consider the
weighted  character sums
$$
S_q(\chi;h;\cA) = \sum_{n=1}^h a_n\chi(n) .
$$ 
If $a_n=1$ for all $n$, we simply use the notation 
$$
S_q(\chi;h) = \sum_{n=1}^h \chi(n) .
$$

First we recall that by the
P{\'o}lya-Vinogradov  (for $\nu =1$) and Burgess
(for $\nu\ge2$)  bounds,
see~\cite[Theorems~12.5 and~12.6]{IwKow}, for an arbitrary  integers
$q \ge h\ge 1$, the bound
\begin{equation}
\label{eq:PVB} 
\max_{\chi \in \cX_q \backslash \{\chi_0\}}
\left|S_q(\chi;h) \right|  \le h^{1 -1/\nu} q^{(\nu+1)/4\nu^2 + o(1)}
\end{equation}
holds with $\nu = 1,2,3$ for any $q$ and with an arbitrary
positive integer $\nu$ if $q$ is cube-free. 

It is well-known that  
assuming the Generalized Riemann Hypothesis (GRH), we derive a ``square-root cancellation'' bound
\begin{equation}
\label{eq:SQRC} 
\max_{\chi \in \cX_q \backslash \{\chi_0\}}
\left|S_q(\chi;h) \right|  \le h^{1/2} q^{o(1)},
\end{equation} 
 and in particular is quoted
in~\cite[Bound~(13.2)]{Mont}. Despite this, it seems to be
difficult to find a proof of this bound, however  one 
can easily derive it from~\cite[Theorem~2]{GrSo}.

%

%


%

Furthermore, we use the following well-known property of the Gauss sums
$$
\tau_q(\chi) = \sum_{v=1}^q \chi(v) \e(v/q), \qquad \chi\in \cX_q,
$$
see, for example,~\cite[Equation~(3.12)]{IwKow}.

\begin{lemma}
\label{lem:tau chi}
For any primitive multiplicative character $\chi \in \cX_q^*$ and an integer $b$ with
$\gcd(b,q) = 1$,  we have
$$
\chi(b) \tau_q( \overline\chi) =
\sum_{\substack{v=1\\ \gcd(v,q) =1}}^{q}  \overline\chi(v) \e(bv/q),
$$
where $\overline\chi$ is the complex conjugate character to $\chi$.
\end{lemma}

By~\cite[Lemma~3.1]{IwKow} we also have:

\begin{lemma}
\label{lem:tau size}
For any  $\chi\in \cX_q^*$ we have
$$
|\tau_q(\chi)| = q^{1/2}.
$$
\end{lemma}

We also   recall the classical large sieve
inequality, see~\cite[Theorem~7.11]{IwKow}:

\begin{lemma}
\label{lem:Large Sieve}
Let $a_1, \ldots, a_T$ be an arbitrary sequence
of complex numbers and let
$$
A = \sum_{n=1}^T |a_n|^2 \mand
T(u) = \sum_{n=1}^T a_n \exp(2 \pi i n u).
$$
Then, for an arbitrary integer  $Q\ge 1$, we have
$$
\sum_{q=1}^Q \sum_{\substack{v=1\\ \gcd(v,q) =1}}^{q}
\left|T(v/q)\right|^2 \ll
\(Q^2+ T\) A.
$$
\end{lemma}

\section{Bounds of character sums for almost all moduli}
\label{sec:bound char}

Garaev~\cite{Gar1},  has obtained a series of improvements
of the bound~\eqref{eq:PVB}  which hold  for almost all 
moduli integer $q\ge 1$.
Namely, by~\cite[Theorem~10]{Gar1}, for any $\delta < 1/4$ if $h$ and $Q$
tend to infinity in such a way that
$$
\frac{\log h}{\sqrt{\log Q}}\to \infty
$$
then the bound
$$
\max_{\chi\in \cX_q^*}   \left|\sum_{n=1}^h \chi(n) \right|\le h^{1-\delta}
$$
holds for all but at most  $Q^{4\delta} h^{(1-2\delta) \gamma +o(1)}$
moduli $q \le Q$, where $\gamma$ is the following fractional parts  
\begin{equation}
\label{eq:gamma h}
\gamma = \left\{\frac{2\log Q}{\log h}\right\}.
\end{equation}

 Here we give some modifications
of the bounds from~\cite{Gar1} which are more convenient for our applications.
In particular, the size of the exceptional set in~\cite[Theorem~10]{Gar1}
of moduli $q \le Q$ for which 
depends on the fractional part $\gamma$.

One can simply estimate $\gamma \le 1$ and still derive a
nontrivial bound $O(Q^{4\delta} h^{1-2\delta})$
from~\cite[Theorem~10]{Gar1}. However here we
show that one can modified the argument of Garaev~\cite{Gar1}
and obtain a stronger bound than that corresponds to replacing $\gamma$
with 1. We also show that the argument of~\cite{Gar1} augmented by
some standard techniques, can be
used to estimate the largest values of sums  $|S_q(\chi;h)|$
uniformly over all integers $h \le H$ and $\chi\in \cX_q^*$, which is important for some
applications. 

We now define $\gamma$   by the analogue of~\eqref{eq:gamma h} but with $H$ 
instead of $h$, that is, 
\begin{equation}
\label{eq:gamma H}
\gamma = \left\{\frac{2\log Q}{\log H}\right\}.
\end{equation}

\begin{lemma}
\label{lem:Almost all qH} Let $H$ and $Q$  be sufficient large positive 
integer numbers with $Q \ge H\ge Q^\varepsilon$ for some fixed $\varepsilon > 0$
and let $\cA=(a_n)_{n\in\mathbb{N}}$ be an arbitrary sequence of complex numbers
with $|a_n|=1$. 
Then for any $\delta< 1/4$ the bound
$$
\max_{\chi\in \cX_q^*} \max_{h \le H}  \left|S_q(\chi;h;\cA) \right|\le H^{1-\delta}
$$
holds true for all but at most 
$Q^{4\delta} H^{\vartheta+ o(1)}$ moduli $q \le Q$,
where $\gamma$ is given by~\eqref{eq:gamma H} and
$\vartheta = \min\{ (1-2 \delta) \gamma, 2 \delta(1-\gamma)\}$
\end{lemma}

\begin{proof} As we have mentioned, we follow the ideas of 
Garaev~\cite[Theorem~3]{Gar1}.

Without loss of generality we may assume that $H=2M+1$ is an
odd integer.
We also define the function $e(z) = \exp(2 \pi i z)$.
We recall, that for any integer $z$,  we have
the orthogonality relation
\begin{equation}
\label{eq:Orth}
\sum_{b=-M}^M \e(bz/H) = \left\{\begin{array}{ll}
H,&\quad\text{if $z\equiv 0 \pmod H$,}\\
0,&\quad\text{if $z\not\equiv 0 \pmod H$,}
\end{array}
\right.
\end{equation}
see~\cite[Section~3.1]{IwKow}.
We also need the bound
\begin{equation}
\label{eq:Incompl}
\sum_{n=u+1}^{u+h} \e(bn/H) \ll  \frac{H}{|b|+1},
\end{equation}
which holds for any integers  $b$, $u$ and $H\ge h\ge 1$ with $|b| \le H/2$,
see~\cite[Bound~(8.6)]{IwKow}.

Now for each $q \le Q$ we fix  $\chi_q \in \cX_q^*$ and  $h_q\le H$
with
$$
 \left|S_q(\chi_q;h_q;\cA)  \right|
= \max_{\chi\in \cX_q^*} \max_{h \le H} \left|S_q(\chi;h;\cA)\right|.
$$
Then using~\eqref{eq:Orth},  we write
\begin{eqnarray*}
S_q(\chi_q;h_q;\cA)  &=&
 \sum_{r=1}^H  a_r \chi_q(r)\frac{1}{H} \sum_{n=1}^{h_q} \sum_{b=-M}^{M} \e(b(r-n)/H) \\
&=& \frac{1}{H} \sum_{b=-M}^{M}   \sum_{n=1}^{h_q} \e(-bn/H)
 \sum_{r=1}^H   a_r \chi_q(r) \e(br/H).
\end{eqnarray*}
Recalling~\eqref{eq:Incompl}, we see that
$$
S_q(\chi_q;h_q;\cA) \ll
 \sum_{b=-M}^{M} \frac{1}{|b|+1}
\left|\sum_{r=1}^H   a_r \chi_q(r) \e(br/H)\right| .
$$
Writing
$$
|b|+1 = \(|b|+1\)^{(2\nu-1)/2\nu} \(|b|+1\)^{1/2\nu},
$$
and using the H{\"o}lder inequality, we derive
 \begin{equation}
\label{eq:Ub}
\sum_{q \le Q} \left|S_q(\chi_q;h_q;\cA)\right|^{2\nu} \ll
(\log Q)^{2\nu-1} \sum_{b=-M}^{M} \frac{1}{|b|+1}  U_b,
\end{equation}
where
$$
U_b =  \sum_{q \le Q} \left|\sum_{r=1}^H   a_r \chi_q(r) \e(br/H)\right|^{2\nu}.
$$
We now note that
$$
\(\sum_{r=1}^H   a_r \chi_q(r) \e(br/H)\)^\nu = \sum_{n=1}^{T} \rho_{b}(n)   \chi_q(n) ,
$$
 where $T = H^\nu$ and
$$
\rho_{b}(n) = \sum_{\substack{r_1,\ldots, r_{\nu}=1\\ r_1\ldots r_{\nu} = n}}^H
a_{r_1} \ldots a_{r_\nu}
\e(b( r_1+\ldots+ r_\nu)/H).
$$
Using Lemma~\ref{lem:tau chi}, we  write
\begin{align*}
\(\sum_{r=1}^H   a_r \chi_q(r) \e(br/H)\)^\nu &= \sum_{n=1}^{T} \rho_{b}(n)
 \frac{1}{\tau_{q}( \overline \chi_q)}\sum_{\substack{v=1\\ \gcd(v,q) =1}}^{q}
  \overline \chi_q(v) \e(nv/q)\\
&= \sum_{\substack{v=1\\ \gcd(v,q)=1}}^q \frac {\overline \chi_q(v)}{\tau_{q}(\overline \chi_q)} 
\sum_{n=1}^T \rho_b(n) e(nv/q).
\end{align*}
Changing the order of summation,
by Lemma~\ref{lem:tau size} and the Cauchy inequality, we obtain,
$$
\left|\sum_{r=1}^H  \chi_q(r) \e(br/H)\right|^{2\nu} \le
\sum_{\substack{v=1\\ \gcd(v,q) =1}}^{q}  \left| \sum_{n=1}^{T} \rho_{b}(n) \e(nv/q)\right|^2.
$$
Therefore,
$$
U_b \le   \sum_{q \le Q} \sum_{\substack{v=1\\ \gcd(v,q) =1}}^{q} \left| \sum_{n=1}^{T} \rho_{b}(n) \e(nv/q)\right|^2.
$$
Recalling the well-known upper bound on the divisor function $d(n)$, 
see~\cite[Bound~(1.81)]{IwKow}, 
we conclude that
$$
|\rho_{b}(n)| \le  \sum_{\substack{r_1, \ldots, r_\nu =1\\r_1 \ldots r_\nu=n}}^H 1
\le (d(n))^\nu =  n^{o(1)}
$$
as $n \to \infty$. Thus
$$
\sum_{n=1}^T |\rho_b(n)|^2 \le T^{o(1)} \sum_{n=1}^T |\rho_b(n)| \le T^{o(1)} H^\nu = H^{\nu(1+o(1))}.
$$
Hence, we now derive from Lemma~\ref{lem:Large Sieve}
$$
U_b\le  \(Q^2+ T \) \sum_{n=1}^T |\rho_b(n)|^2 \le \(Q^2+ H^\nu \) H^{\nu(1+o(1))},
$$
which after substitution in~\eqref{eq:Ub} implies
\begin{equation}
\label{eq:bound nu}
\sum_{q \le Q}  \max_{\chi\in \cX_q^*} \max_{h \le H}  \left|S_q(\chi;h;\cA)\right|^{2\nu} \le
\(Q^2+ H^\nu \) H^{\nu(1+o(1))}.
\end{equation}

 We now define
 the integer $k$ by
 $$
k = \fl{\frac{2\log Q}{\log H}}.
$$
Note that
$$
Q^2 = H^{k+\gamma}.
$$
Using~\eqref{eq:bound nu} with $\nu=k$ (so $\nu<2/\epsilon$ in particular) we see that
$$
\sum_{q \le Q}  \max_{\chi\in \cX_q^*} \max_{h \le H}  \left|S_q(\chi;h;\cA)\right|^{2k} \le
Q^{2} H^{k+o(1)}.
$$
Hence the desired bound holds
for all but at most
\begin{equation}
\label{eq:small gamma}
\begin{split}
Q^{2} H^{k+o(1)} H^{-2k(1-\delta)}& = Q^{2} H^{-k(1-2\delta)+o(1)}
= H^{2k\delta + \gamma+o(1)}\\
& = Q^{4\delta } H^{(1-2 \delta) \gamma+o(1)}
\end{split}
\end{equation}
moduli $q \le Q$ (which is essentially a bound
of the same strength as that of~\cite[Theorem~10]{Gar1}).

Furthermore, using~\eqref{eq:bound nu} with $\nu=k+1$ we see that
$$
\sum_{q \le Q}  \max_{\chi\in \cX_q^*} \max_{h \le H}  \left|S_q(\chi;h;\cA)\right|^{2(k+1)} \le
H^{2(k+1)+o(1)} .
$$
Hence the desired  bound  holds
for all but at most
\begin{equation}
\label{eq:big gamma}
H^{2(k+1)+o(1)} H^{-2(k+1)(1-\delta)} = H^{2(k+1)\delta +o(1)}\\
 = Q^{4\delta} H^{2 \delta(1-\gamma)+ o(1)}
\end{equation}
moduli $q \le Q$.

The bounds~\eqref{eq:small gamma} and~\eqref{eq:big gamma}
yield  the result.
\end{proof}

Covering the interval $[1,H]$ by $O(\log H)$ dyadic intervals of the form 
$[H_0/2, H_0]$, and using that 
$$
\min\{ (1-2 \delta) \gamma, 2 \delta(1-\gamma)\}  \le 2\delta(1-2\delta),
$$
we obtain:

\begin{cor}
\label{cor:Almost all qhH} Let $H$ and $Q$  be sufficient large positive 
integer numbers with $Q \ge H\ge Q^\varepsilon$ for some fixed $\varepsilon > 0$ 
and let $\cA=(a_n)_{n\in\mathbb{N}}$ be an arbitrary sequence of complex numbers
with $|a_n|=1$.
Then for any $\delta< 1/4$ the bound
$$
\max_{\chi\in \cX_q^*}   \left|S_q(\chi;h;\cA)  \right|\le h^{1-\delta}
$$
holds true for all $h \le H$ and for all but at most 
$Q^{4\delta} H^{2\delta(1-2\delta)+ o(1)}$ moduli $q \le Q$.
\end{cor}

%
%

%
%

For the traditional character sums, that is, if $a_n=1$,
we also have the following result.

\begin{cor}
\label{cor:Almost all qh} Let $Q$  be a sufficient large positive 
integer number. For any fixed $\varepsilon > 0$ and $3/14 > \delta > 0$,
there is some $\xi> 0$ such that the bound
$$
\max_{\chi\in \cX_q^*}   \left|S_q(\chi;h) \right|\le h^{1-\delta}
$$
holds true for all $h \in [Q^\varepsilon, Q]$ and for all but at most 
$Q^{1-\xi}$ moduli $q \le Q$.
\end{cor}

\begin{proof} Clearly, it is enough to consider only $q\in [Q/2, Q]$.
 Let us fix some positive $\delta$ with $(3-\sqrt{7})/2 < \delta < 3/14$.
Simple calculus shows that there is some $\alpha> 1/2$ such that 
$$
4\delta +2\alpha \delta(1-2\delta) < 1
\mand
4\delta +(2-3\alpha) (1-2\delta) < 1.
$$
We now note that with the above parameters,  Corollary~\ref{cor:Almost all qhH},
used with $H = \rf{Q^\alpha}$, 
implies that it remains 
to establish the results only for the values of $h \in [Q^\alpha, Q]$.

Furthermore, by   the P{\'o}lya-Vinogradov   bound (that is, by~\eqref{eq:PVB}
taken with $\nu =1$) we have
$$
\max_{\chi\in \cX_q^*}   \left|S_q(\chi;h) \right|\le h^{1-\delta}
$$
holds for any $h \ge Q^{1/2(1-\delta)}$ and $q \le Q$.

Therefore, we only need to consider the values of $h$ 
in the interval $[Q^\alpha, Q^{1/2(1-\delta)}]$, which we
can cover by $O(\log Q)$ dyadic intervals $[H/2, H]$.
Now, for  $H \in [Q^\alpha, Q^{1/2(1-\delta)}]$ we have
$$
3 < 4(1 -\delta) \le \frac{2\log Q}{\log H} \le 2 \alpha^{-1} < 4.
$$
Hence, writing $H = Q^\beta$, for the parameter  $\gamma$, 
that is given by~\eqref{eq:gamma H}, we have
$$
\gamma =2\beta^{-1}-3. 
$$
Recalling Lemma~\ref{lem:Almost all qH}, we see that it remains to check that 
$$
4\delta + \beta \min\{(2\beta^{-1}-3)(1-2 \delta), 2 (4-2\beta^{-1})\delta\}<1
$$ 
for every $\beta\in [\alpha, 1/2(1-\delta)]$.
We now have the following elementary estimates
\begin{equation*}
\begin{split}
4\delta  + \beta &\min\{(2\beta^{-1}-3)(1-2 \delta), 2 (4-2\beta^{-1})\delta\}\\
& = 4\delta + \beta (2\beta^{-1}-3)(1-2 \delta)
=  4\delta +  (2-3\beta)(1-2 \delta)\\
& 
\le   4\delta +  (2-3\alpha)(1-2 \delta) < 1
\end{split}
\end{equation*}
and the result follows. 
\end{proof}

\section{Background from Additive Combinatorics}

We use standard notation of additive combinatorics,
including sumsets $\cA+\cB = \{a+b~:~a \in \cA,\ b \in \cB\}$ 
and $k$-folded sumsets $k\cA = \{a_1+\ldots+a_k~:~a_1,\ldots,a_k\in \cA\}$, 
assuming that $\cA$ and $\cB$ are subsets of some
abelian group $\cG$. 

We first recall the {\it Pl{\"u}nnecke inequality\/}, 
see~\cite[Corollary~6.29]{TaoVu}.

\begin{lemma}
\label{lem:PlunIneq}
Suppose that $\cA$ and $\cB$ are subsets of some
abelian group $\cG$, and that $\#(\cA +\cB) \le K\#\cA$ 
for some $K \ge 1$. Then for any nonnegative integers $k$ and 
$m$ we have
$$
\#(k\cB - m\cB) \le K^{k+m}\#\cA.
$$ 
\end{lemma}

We now record the following obvious  consequence of  Lemma~\ref{lem:PlunIneq}.
\begin{cor}
\label{cor:PowerApproxSubgr}
For any fixed integer $m\ge 1$ and approximate
subgroup  $\cU \subseteq \F_p^*$ we have
$$
\# \cU^{(m)} \le (\# \cU)^{1+o(1)}.
$$
\end{cor}

Suppose that $\cA \subseteq \cG$  and $\cB \subseteq \cH$  are subsets of abelian groups
$\cG$ and $\cH$, respectively. A map $\psi: \cA \to \cB$ is called
{\it Freiman $k$-homomorphism\/} if whenever
$$
a_1+\ldots+a_k = a_{k+1} + \ldots + a_{2k}
$$ 
for some $a_1,\ldots,a_{2k}$ then we also have 
$$
\psi(a_1)+\ldots+\psi(a_k) = \psi(a_{k+1}) + \ldots + \psi(a_{2k}).
$$ 
If $\psi$ has an inverse which is also a Freiman $k$-homomorphism 
then we say that  $\psi$ is a
{\it Freiman $k$-isomorphism\/} and also that  $\cA$ and $\cB$ are 
{\it Freiman $k$-isomorphic\/}. 

We note that if $\cG$ is a  torsion-free group 
that considering $a_1=\ldots=a_k = a$ and $a_{k+1}=\ldots =a_{2k}=b$
for some $a,b \in\cA$ we derive that any Freiman $k$-isomorphism
is an injection.

We need the following result of Ruzsa~\cite[Theorem~2.3.5]{Ruz2},  which is 
known as the {\it Modelling Lemma\/} (see also~\cite[Theorem~2]{Ruz1}
for teh case $\cG = \Z$ which is fully sufficient for our 
purposes).

\begin{lemma}
\label{lem:RuzsaModel}
Suppose that  $\cA \subseteq \cG$ 
is a finite nonempty subset of a  torsion-free Abelian group $\cG$. 
 Then for  all integers $k \ge 2$ and  $q \ge |k\cA - k\cA|$ there is a set 
$\cB \subseteq \cA$ with 
$\# \cB \ge \# \cA/k$ such that $\cB$ is Freiman $k$-isomorphic to a subset of 
$\Z/q\Z$.
\end{lemma}

We now use  Lemma~\ref{lem:RuzsaModel} to show that sets with a small
doubling contain subsets of a give cardinality and also with small
doubling. We present it in a more general and explicit form 
than we need for applications, as we think it maybe of independent 
interest. 

\begin{lemma}
\label{lem:SmallDouble}
Suppose that  $\cA \subseteq \cG$ 
is a finite nonempty subset of a  torsion-free Abelian group $\cG$
of cardinality $N =\#\cA$
such that for some $L\ge 1$ we have $\#(2\cA) \le LN$.
Then for any positive integer $M\le N$ 
there is a set 
$\cC \subseteq \cA$ with 
$$
\# \cC = M \mand \#(2\cC) \le 10L^4M.
$$
\end{lemma}

\begin{proof} If $M \ge N/2$ we simply take $C$ 
to be any subset of $\cA$ of cardinality $M$.
Then 
$$
\#(2\cC)  \le  \#(2\cA) \le LN \le 2LM.
$$

Now assume that $M \le N/2$. 
First we note that applying Lemma~\ref{lem:PlunIneq}, we derive 
$\#(2\cA - 2\cA) \le K N$, where $K = L^4$. 

Let 
$$\cB \subseteq \cA \mand KN \le q \le 2KN
$$ 
be as in Lemma~\ref{lem:RuzsaModel}
(applied with $k=2$)
and let $\psi$ be the corresponding Freiman $2$-isomorphism.
We consider the set $\cX = \psi(\cB) \subseteq \Z/q\Z$.
As we have noticed, $\psi$ is an injection, so 
$$
\# \cX  = \# \cB\ge N/2 \ge M.
$$ 
By a simple averaging argument, for any integer  $R\ge 1$ there is a 
subset   $\cY \subseteq \Z/q\Z$ of $R$ consecutive residue classes modulo $q$, 
that is,    of $\{r, \ldots, r+R-1\}$ for some $r \in \Z$
and such that 
$$
\#\(\cX \cap \cY\) \ge  \frac{\#\cX \cdot \#\cY}{q} = \frac{R}{q}\#\cX.
$$
We now take 
$$
R = \rf{\frac{qM}{\#\cX}}
$$
to guarantee $\#\(\cX \cap \cY\)\ge M$.
 We now collect arbitrary $M$ elements of  
$\cX \cap \cY$ in one set $\cZ$ and define 
$$
\cC = \psi^{-1}(\cZ).
$$
We clearly have $\# \cC = \# \cZ = M$ and
also by the property of Freiman $2$-isomorphisms
$$
 \#(2\cC)  =  \#(2\cZ) \le \#(2\cY)  \le  2\#\cY = 2R
$$
(since $\cY$ consists of consecutive residue classes). 
Furthermore, we have
$$
R  \le   \rf{\frac{qM}{\#\cX}} \le
   \rf{2qM/N}  \le  \rf{4KM} =  \rf{4L^4M}
 \le 5L^4M
$$
which concludes the proof. 
\end{proof}

We now  see that Lemma~\ref{lem:SmallDouble} implies 
that an
approximate subgroup of $\F_p^*$ contains subsets of any size 
that   behave as approximate subgroups.

\begin{lemma}
\label{lem:Subset AprSubgr}
For any approximate subgroup  $\cU \subseteq \F_p^*$, for any integer $M \le \# \cU$ 
one can find a subset $\cV \subseteq \cU$ such that $\# \cV=M$ and 
$$
\# \cV^{(2)} \le  \# \cV (\# \cU)^{o(1)},
$$
\end{lemma}

\begin{proof}We fix a primitive root $g$ of $\F_p^*$ 
and define the set 
$$
\cA = \{a\in \{0, \ldots, p-2\}~:~ g^a  \in \cU\}.
$$ 
We consider $\cA$ as the set of integers and since $0 \le a+b \le 2p-4$, 
at most two elements from $2A$ correspond to the same element 
in $\cU^{(2)}$. So,  we conclude that 
$$
\#(2A) \le 2 \# (\cU^{(2)}).
$$
The result now follows immediately from Lemma~\ref{lem:SmallDouble}. 
\end{proof}

We note that in our applications of Lemma~\ref{lem:Subset AprSubgr}
the sets $\cU$ and $\cV$ are 
of comparable cardinalities so $(\# \cU)^{o(1)} = (\# \cV)^{o(1)}$
so $\cV$ is also an approximate subgroup. 

\section{Some Equation over $\F_p$ with Variables from 
Intervals and Subgroups}

One easily verifies that Corollary~\ref{cor:PowerApproxSubgr} allows 
us to obtain the 
following  slight variation  of~\cite[Theorem~1]{CillGar2} where instead 
of the sets  $\cU \subseteq \F_p$ with $\# \cU^{(2)} \le 10\# \cU$ we use
approximate subgroups. The proof then goes through without any changes. 

\begin{lemma}
\label{lem:eq2var}
Let an initial interval $\cI \subseteq \F_p$ of length $H$ and an approximate
subgroup  $\cU \subseteq \F_p^*$ of size $N$ satisfy
$$H^{k}N < p \mand  N \le p^{k/(2k+1)}
$$
for some fixed integer $k\ge 1$.
Then the number $J$ of solutions of the equation over $\F_p$
$$
x_1 = x_2u, \qquad u\in \cU,\  x_1, x_2 \in \cI, 
$$
satisfies 
$$
J \le H N^{o(1)} . 
$$
\end{lemma}

Accordingly, we also have the following version 
of~\cite[Corollary~1]{CillGar2}:

\begin{cor}
\label{cor:eq2var}
Let an initial interval  $\cI \subseteq \F_p$ of length $H$ and an approximate
subgroup  $\cU \subseteq \F_p^*$ of size $N$ satisfy
$$H^{k}N < p \mand  N \le p^{k/(2k+1)}
$$
for some fixed integer $k\ge 1$.
Then the number $K$ of solutions of the equation over $\F_p$
$$
x_1u_1 = x_2u_2,  \qquad u_1,u_2\in \cU,\ x_1, x_2 \in \cI, 
$$
satisfies 
$$
K \le H N^{1+o(1)} .
$$
\end{cor}

 We now prove the following direct extension of~\cite[Lemma~7]{CillGar2}:

\begin{lemma}\label{lem: eq3var}
Let an initial interval $\cI \subseteq \F_p$ of length $H$ and an approximate
subgroup  $\cU \subseteq \F_p^*$ of size $N$ satisfy
$$H \le N/2, \qquad H^kN < p, \qquad  N \le p^{k/(2k+1)}
$$
for some fixed integer $k\ge 1$ and let $\cQ$ be the set of 
primes $q \in [N/2, N]$. 
Then the number $S$ of solutions of the equation over $\F_p$
$$
q_1u_1 x_1 =q_2u_2 x_2, \qquad  q_i \in Q, \ u_i  \in \cU,\  x_i \in \cI, \quad i=1,2,
$$
satisfies 
$$
S \le H N^{2+o(1)}. 
$$
\end{lemma}

\begin{proof}
We have $S=S_1+S_2$ where $S_1$ is the number of solutions 
with the additional condition $q_1=q_2$, and $S_2$ is the number of solutions with
$q_1\ne q_2$. We observe that 
Hence, we can apply Corollary~\ref{cor:eq2var} and derive
\begin{equation}
\label{eq:S1}
S_1 \le H N^{2+o(1)}.
\end{equation} 

It remains to estimate $S_2$, we fix $x_2,u_1,u_2$ such that
for $\lambda = u_2x_2/u_1$ we have 
\begin{equation}
\label{eq:S2 T2}
S_2\leq H N^2  T_2,
\end{equation} 
where $T_2$ is the number of solutions of the  equation
$$\frac{q_1x_1}{ q_2} = \lambda, \qquad  q_1,q_2 \in Q, \ q_1 \ne q_2, \ x_1 \in \cI.
$$
From $H<N/2$, we deduce that  $\gcd(q_1x_1,q_2)=1$. Since $N^2 H<p$, 
from~\cite[Lemma~3]{CillGar2} we 
derive  that $x_1q_1$ and $q_1$ are uniquely determined. Since $x_1<q_1$, the value $x_1q_1$ uniquely 
determines $x_1$ and $q_1$. Hence, $T_2\leq 1$,  which together with~\eqref{eq:S2 T2}
implies  
\begin{equation}
\label{eq:S2}
S_2 \le H N^{2}.
\end{equation} 
Combining~\eqref{eq:S1} and~\eqref{eq:S2}, we 
conclude the proof.
\end{proof}

\section{Products of Intervals and Subgroups}

Following the standard notation we use 
$$
\cA\cdot \cB = \{ab~:~a\in \cA, \ b \in \cB\}
$$
to denote the product set of two sets $\cA, \cB \in \F_p$. 
 
We say that a certain property holds  for almost all primes $p$, 
if  it fails for $o(Q/\log Q)$ primes $p\le Q$ as $x \to \infty$. 
 
Here we are interested in the cardinality of the set $\cI \cdot \cU$ for 
 an initial interval  $\cI \subseteq \F_p$  and an approximate
subgroup  $\cU \subseteq \F_p^*$. In particular,  for almost all primes $p$, 
we extend~\cite[Theorem~3]{CillGar2} to a wider range of $\#\cI$ 
and $\# \cU$.

\begin{theorem} 
\label{thm:IU=p}
For any fixed $\alpha$ with  $1/3 \le \alpha < 1/2$ and $\kappa> 0$, for almost all 
primes $p$,  
for any initial interval $\cI \subseteq \F_p$ of length $H$ and   approximate
subgroup  $\cU \subseteq \F_p^*$ of size $N$ that  satisfy
$$
H >p^{1- \alpha +\kappa} \mand N \ge p^{\alpha},
$$
we have 
$$\#\(\cI \cdot \cU\) = p + O(p^{1-\eta}), 
$$
where 
$$
\eta = \frac{3\kappa}{7(1 + \kappa)}.
$$ 
\end{theorem}

%

\begin{proof}  Let $Q$ be a sufficiently large positive integer. 
It is clear that it is enough to establish the 
desired result for all but $o(Q/\log Q)$ primes $p$ in the dyadic
interval $p \in [Q/2, Q]$. Using Corollary~\ref{cor:Almost all qh}
with some fixed positive $\varepsilon < 1-2 \alpha$ and $\delta < 3/14$ 
 we see 
that we can remove  $o(Q/\log Q)$ primes $p \in [Q/2, Q]$
such that for remaining primes $p$ we have 
\begin{equation}
\label{eq:bound p}
\max_{\chi\in \cX_p^*}   \left|\sum_{n=1}^h \chi(n) \right|\le h^{1-\delta} 
\end{equation}
for every integer 
\begin{equation}
\label{eq:h Interv}
h \in [p^{\varepsilon}, p].
\end{equation}
provided that $Q$ is large enough. 

We now always assume that $p$ is such that~\eqref{eq:bound p}
holds.

We now set
$$
m = \rf{\kappa^{-1}}, \quad \ell=\fl{p^{1/m}}, \quad   M = \fl{p^{\alpha}}, 
\quad h = \fl{0.4 p^{1-2 \alpha}} .
$$

By Lemma~\ref{lem:Subset AprSubgr}, we can choose a subset $\cV \subseteq \cU$ 
such that 
$$
\#\cV = M \mand \# \cV^{(2)} \le  \# \cV p^{o(1)} =  \(\# \cV\)^{1+ o(1)}  .
$$

Let $\cQ$ be the set of primes $q \in [M/2, M]$.

One verifies that 
$$
h \ell M \le 0.4 p^{1-2 \alpha} \times    p^{1/m}  \times p^{\alpha} 
=  0.4 p^{1-  \alpha+ 1/m} \le H
$$
as $1/m < \kappa$,
Hence it suffices to prove that for some $\rho >0$ that depends only on $\alpha$, $\kappa$ and $\varepsilon$, 
there are at most $O(p^{1-\rho})$ values of $\lambda \in \F_p^*$  for which the equation 
over $\F_p$ 
\begin{equation}
\label{eq:qvxz}
qvxz= \lambda
\end{equation}
has no solution in $q \in \cQ$, $v \in \cV$ and  positive integers $x \le h$, $z\le \ell$.

Let $\Lambda \subset \F_p^*$ be the set of this elements $\lambda$
and let $L = \# \Lambda$.

We use the orthogonality of characters $\chi\in \cX_p$
to express the number of solutions to~\eqref{eq:qvxz}
for $\lambda\in \Lambda$ via the following character sums:
$$
\frac{1}{p-1} \sum_{\lambda \in \Lambda} 
\sum_{q \in \cQ}\sum_{v \in \cV}  \sum_{x\leq h}\sum_{z\leq \ell} \sum_{\chi \in \cX_p}
\chi(qvxz\lambda^{-1}) = 0.
$$

We now clear the denominator, change the order of summations and 
separate the term corresponding to the principal character $\chi=\chi_0$. This leads us to the equation
$$
h\ell L M \#\cQ  + \sum_{\chi \in \cX_p^*}
\sum_{x\leq h}\sum_{q \in \cQ}\sum_{v \in \cV}\chi(qvx)  \sum_{z\leq \ell }\chi(z) 
\sum_{\lambda\in\Lambda}\chi(\lambda) = 0. 
$$
Therefore 
\begin{equation}
\label{eq:sepchar}
h\ell LM \#\cQ  \le W.
\end{equation}
where 
$$
W = \sum_{\chi \in \cX_p^*} 
\left|\sum_{x\leq h}\sum_{q \in \cQ}\sum_{v \in \cV}\chi(xqv)\right|
\left|\sum_{z\leq \ell }\chi(z)\right| \left|\sum_{\lambda\in\Lambda}\chi(\lambda)\right|.
$$
Because $\varepsilon < 1-2 \alpha$, if $Q$ is sufficiently large,  
the condition~\eqref{eq:h Interv} is satisfied for the above 
choice of $h$.  Therefore, the bound~\eqref{eq:bound p} holds and we write
$$
\left|\sum_{x\leq h}\sum_{q \in \cQ}\sum_{v \in \cV}\chi(xqv)\right|
\le \(h^{1-\delta}M\#\cQ\)^{1/m} 
\left|\sum_{x\leq h}\sum_{q \in \cQ}\sum_{v \in \cV}\chi(xqv)\right|^{(m-1)/m}.
$$

Using the fact that 
$$
\frac{m-1}{2m}+\frac{1}{2m}+\frac{1}{2}=1,
$$ 
and extending the summation over all $\chi \in \cX_p$
we 
obtain 
\begin{equation}
\begin{split}
\label{eq:bound W}
W&\le \(h^{1-\delta}M\#\cQ\)^{1/m}  \(\sum_{\chi \in \cX_p}
\left|\sum_{x\leq h}\sum_{q \in \cQ}\sum_{v \in \cV}\chi(xqv)\right|^{2}\)^{(m-1)/2m}\\
&\qquad \qquad \quad  \(\sum_{\chi \in \cX_p}\left|\sum_{z\leq \ell}\chi(z)\right|^{2m}\)^{1/2m}
\(\sum_{\chi \in \cX_p}\left|\sum_{\lambda \in \Lambda}\chi(\lambda)\right|^{2}\)^{1/2}.
\end{split}
\end{equation}

First, using the orthogonality of characters, we obtain
\begin{equation}
\label{eq:bound 1}
\sum_{\chi \in \cX_p}\left|\sum_{\lambda \in \Lambda}\chi(\lambda)\right|^{2}=(p-1)L
\end{equation}
and
$$
\sum_{\chi \in \cX_p}\left|\sum_{z\leq \ell}\chi(z)\right|^{2m}=(p-1)S, 
$$
where $S$ is the number of solutions of the following equation over $\F_p$:
$$
z_1\cdots z_m =  z_{m+1}\cdots z_{2m}, \qquad  1\le z_j\le \ell, \ i =1, \ldots, 2m.
$$
Since $L^{m}<p$, this is in fact equation over $\Z$ and from the well-known bounds
of the divisor function, we obtain $ S \le \ell^{m+o(1)}$ solutions. 
Hence, we have 
\begin{equation}
\label{eq:bound 2}
\sum_{\chi \in \cX_p}\left|\sum_{z\leq \ell}\chi(z)\right|^{2m} \le p \ell^{m+o(1)}.
\end{equation}

Furthermore, the same orthogonality 
property implies  that 
\begin{equation}
\label{eq:sum T}
\sum_{\chi \in \cX_p}\left|\sum_{x\leq h}\sum_{q \in \cQ}
\sum_{v \in \cV}\chi(qvx)\right|^{2}=(p-1)T,
\end{equation}
where $T$ is the number of solutions of the following equation over $\F_p$
\begin{equation}
\label{eq:qvx}
q_1v_1 x_1 =q_2v_2 x_2, \qquad  q_i \in Q, \ v_i  \in \cV,\  1\le x_i \le h, \quad i=1,2.
\end{equation}
Using $\alpha \ge 1/3$, one verifies that for any $k\ge 2$ and a sufficiently large $Q$, we have
$$
h  \le M/2.
$$
Furthermore, if we define an integer $k\ge 1$ by the inequalities
$$
\frac{k-1}{2k-1} < \alpha < \frac{k}{2k+1}.
$$
then we have 
$$
M \le  p^{\alpha} \le p^{k/(2k+1)}
$$
and 
$$
h^kM  < p^{k(1-2\alpha) + \alpha}  = p^{k-(2k-1)\alpha}  < p.
$$
Hence, due to 
the choice of $\cV$, we see that  Lemma~\ref{lem: eq3var} applies to 
the  equation~\eqref{eq:qvx} and implies $T \le h M^{2 + o(1)}$, which together 
with~\eqref{eq:sum T} yields
\begin{equation}
\label{eq:bound 3}
\sum_{\chi \in \cX_p}\left|\sum_{x\leq h}\sum_{q \in \cQ}
\sum_{v \in \cV}\chi(qvx)\right|^{2} \le p h M^{2 + o(1)}. 
\end{equation}

Substituting~\eqref{eq:bound 1}, \eqref{eq:bound 2} and~\eqref{eq:bound 3}
in~\eqref{eq:bound W} and recalling~\eqref{eq:sepchar}, we obtain 
$$h\ell LM \#\cQ  \le \(h^{1-\delta}M\#\cQ\)^{1/m} (p\ell^{m})^{1/2m}(pL)^{1/2}
(p h M^{2 + o(1)})^{(m-1)/2m}.$$ 
Since $\# \cQ= M^{1+o(1)}$, we obtain
$$
h \ell LM^2  \le  h^{(m+1)/2m -\delta/m} \ell^{1/2} p L^{1/2} M^{(m+1)/m+ o(1)}
$$
or 
$$
L \le   h^{-2\delta/m} \ell^{-1}  p^{2} (hM^2)^{-1+1/m}.
$$
Finally, since
$$
hM^2 = p^{1+o(1)} 
$$
we derive
$$
L \le   h^{-2\delta/m} \ell^{-1} p^{1+1/m + o(1)}
=  h^{-2\delta/m}   p^{1 + o(1)}.
$$
Recalling the choice of $m$ and $\delta$, we  conclude the 
proof. 
\end{proof}





In the case when $\cU$ is a subgroup of $\F_p^*$, we prove a  more  general and stronger result under the  GRH, 
 which is nontrivial for any $H$ and $N$ as long as 
 $HN>p^{1+\kappa}$ for some fixed $\kappa>0$. 
 
\begin{theorem} 
\label{thm:IU=p-GRH}
Fix $\kappa>0$. Assuming the GRH, for any prime $p$, for any initial interval $\cI \subseteq \F_p$ of length $H$ and subgroup  $\cU \subseteq \F_p^*$ of size $N$ such that $HN>p^{1+\kappa}$, we have
$$\#\(\cI \cdot \cU\) = p + O(p^{1-\kappa + o(1)}).
$$
\end{theorem}

\begin{proof}
It suffices to prove that for some $\rho >0$ that depends only on $\varepsilon$, 
there are at most $O(p^{1-\rho})$ values of $\lambda \in \F_p^*$  for which the equation 
over the field $\F_p$ 
\begin{equation}
\label{eq:ux}
ux= \lambda
\end{equation}
has no solution in  $u \in \cU$ and  positive integers $x \le H$.

Let $\Lambda \subset \F_p^*$ be the set of this elements $\lambda$
and let $L = \# \Lambda$. 

We use the orthogonality of characters $\chi\in \cX_p$
to express the number of solutions to~\eqref{eq:ux}
for $\lambda\in \Lambda$ via the following character sums:
$$
\frac{1}{p-1} \sum_{\lambda \in \Lambda} 
\sum_{u \in \cU}  \sum_{x\leq H}\sum_{\chi \in \cX_p}
\chi(ux\lambda^{-1}) = 0.
$$

As in the proof of Theorem~\ref{thm:IU=p} this leads us to the equation
$$
H L N  + \sum_{\chi \in \cX_p^*}
\sum_{x\leq H}\sum_{u \in \cU}\chi(ux)
\sum_{\lambda\in\Lambda}\chi(\lambda) = 0. 
$$
Therefore 
\begin{equation}
\label{eq:sepcharRH}
H LN  \le W, 
\end{equation}
where 
$$
W = \sum_{\chi \in \cX_p^*} 
\left|\sum_{x\leq h}\sum_{u \in \cU}\chi(xu)\right|
 \left|\sum_{\lambda\in\Lambda}\chi(\lambda)\right|.
$$

Using the Cauchy inequality and extending the summation over all $\chi \in \cX_p$
we obtain 
\begin{equation}
\begin{split}
\label{eq:bound WRH}
W&\le  \(\sum_{\chi \in \cX_p^*}
\left|\sum_{x\leq H}\sum_{u \in \cU}\chi(xu)\right|^{2}\)^{1/2}\(\sum_{\chi \in \cX_p}\left|\sum_{\lambda \in \Lambda}\chi(\lambda)\right|^{2}\)^{1/2}.
\end{split}
\end{equation}

Now we use the fact that 
$$
\sum_{u \in \cU}\chi(u)=0
$$
if $\chi$ is nontrivial over the subgroup $\cU$. Hence there are at 
most $(p-1)/N$ characters such that the above sum does not vanish, 
which case it is equal to $N$.

 Therefore, proceeding as in the proof of Theorem~\ref{thm:IU=p} and using the bound~\eqref{eq:SQRC} we obtain
$$W\le \(\frac{p}{N}(NH^{1/2}p^{o(1)})^{2}\)^{1/2}(pL)^{1/2}.
$$
Substituting in~\eqref{eq:sepcharRH}, yields
$$H LN  \le (pN)^{1/2} H^{1/2} (pL)^{1/2} p^{o(1)},$$ 
which yields the bound
$$L \le \frac{p^{2+o(1)}}{NH}$$ 
that concludes the proof.
\end{proof}

\section{Comments}

Our proof of Corollary~\ref{cor:Almost all qh} 
uses~\eqref{eq:PVB} (with $\nu =1$) and thus does not 
extend to more general weighted sums $S_q(\chi;h;\cA)$.
However, for some interesting sequences $\cA$, that admit a
version of~\eqref{eq:PVB}  one can obtain such a 
result. For example, combining our argument with a
bound of Karatsuba~\cite{Kar}, one can derive a version 
of Corollary~\ref{cor:Almost all qh} 
for the sequence of shifted primes, that is, for the sequence 
$a_n = 1$ if $n= \ell+a$ for a prime $\ell$ and $a_n = 0$ 
otherwise (where $a\ne 0$ is a fixed integer).

We note that we have  slightly modified the scheme of the proof 
of~\cite[Theorem~3]{CillGar2} which has allowed us to extract 
the optimal saving $\eta$ from the preliminary bounds 
used in in the proof of Theorem~\ref{thm:IU=p}.
In particular, instead of separating the sum $W$ 
into contribution from ``good'' and ``bad'' characters
and balancing them, 
we have used a more direct approach via the H{\"o}lder 
inequality, which make the optimal use of bounds on the 
moments of the character sums involved (including 
the ``$\infty$-moment'', that is, the bound on the maximum value
of some of these sums). 

It is easy to see that if for some $p$ instead of~\eqref{eq:SQRC}
we have a weaker bound
$$
\max_{\chi \in \cX_p \backslash \{\chi_0\}}
\left|S_p(\chi;h) \right|  \le h^{1-\delta} p^{o(1)},
$$ 
with some fixed $\delta \le 1/2$, 
the method of proof of 
Theorem~\ref{thm:IU=p-GRH} still applies  and in the case when 
$\cU$ is a subgroup of $\F_p^*$,  leads to 
a nontrivial bound under the condition 
$H^{2 \delta} N>p^{1+\kappa}$. For example, this observation can 
be combined with Corollary~\ref{cor:Almost all qh} to
a nontrivial bound under the condition 
$H^{3/7} N>p^{1+\kappa}$ for almost all $p$. 
On the other hand using the conditional under the GRH 
bound~\eqref{eq:SQRC} in the proof of Theorem~\ref{thm:IU=p}
one can get the same result for all primes and also with 
a larger $\eta = \kappa/(1+\kappa)$.

The question about the set of elements missing from 
the set product $\cI \cdot \cU$, which is considered in 
Theorems~\ref{thm:IU=p} and~\ref{thm:IU=p-GRH} is a
multiplicative version of the question of~\cite{ShpSte}
 about the set of elements missing from 
the set difference $\cI - \cU$  (only in the case 
when $\cU$ is a subgroup of $\F_p^*$). 
The argument of~\cite{ShpSte} also works for the set 
sum $\cI + \cU$ without any changes. However in~\cite{ShpSte}
mostly the case of large subgroups of size $\# \cU > p^{1/2}$
is of interest and so  the technique used is different.

Finally, clearly slightly changing the values of $\eta$ one 
can also include the value $\alpha = 1/2$ in the range 
of Theorem~\ref{thm:IU=p} (for example, one can apply it
with $\alpha =1/2 - \kappa/2$ instead of $1/2$ 
and $\kappa/2$ instead of $\kappa/2$).

\section*{Acknowledgements}

The authors are very grateful to Ben Green for sketching 
them a proof of Lemma~\ref{lem:SmallDouble}. 

 The authors also would like to thank
 CIRM (Luminy) for its support and hospitality during the Research in Pairs 
program in May 2014, where the idea of this work was formed.

During the preparation of this work M.~Munsch was supported by a postdoctoral grant in CRM of Montreal under the supervision of Andrew Granville and Dimitris Koukoulopoulos
and I.~E.~Shparlinski was supported in part by the Australian Research Council
Grant DP140100118.

\end{document}